\documentclass[a4paper,12pt,reqno]{amsart}
\usepackage{amsfonts}
\usepackage{amsmath}
\usepackage{amssymb}
\usepackage[a4paper]{geometry}
\usepackage{mathrsfs}
\usepackage[colorlinks]{hyperref}
\renewcommand\eqref[1]{(\ref{#1})} 
\numberwithin{equation}{section}
\theoremstyle{plain}
\newtheorem{thm}{Theorem}[section]
\newtheorem{prop}[thm]{Proposition}
\newtheorem{cor}[thm]{Corollary}
\newtheorem{lem}[thm]{Lemma}

\theoremstyle{definition}

\begin{document}

   \title[Green's identities on stratified Lie groups]
   {Green's identities, comparison principle and uniqueness of positive solutions
   	 for nonlinear $p$-sub-Laplacian equations on stratified Lie groups}

\author[Michael Ruzhansky]{Michael Ruzhansky}
\address{
  Michael Ruzhansky:
  \endgraf
  Department of Mathematics
  \endgraf
  Imperial College London
  \endgraf
  180 Queen's Gate, London SW7 2AZ
  \endgraf
  United Kingdom
  \endgraf
  {\it E-mail address} {\rm m.ruzhansky@imperial.ac.uk}
  }
\author[Durvudkhan Suragan]{Durvudkhan Suragan}
\address{
  Durvudkhan Suragan:
  \endgraf
  Institute of Mathematics and Mathematical Modeling
  \endgraf
  125 Pushkin str.
  \endgraf
  050010 Almaty
  \endgraf
  Kazakhstan
  {\it E-mail address} {\rm suragan@math.kz}
  }

\thanks{The authors were supported in parts by the EPSRC
 grant EP/K039407/1 and by the Leverhulme Grant RPG-2014-02,
 as well as by the MESRK grant 5127/GF4.  No new data was collected or generated during the course of research.}

     \keywords{$p$-sub-Laplacian, Green's identity, Picone's identity,  D\'iaz-Sa\'a inequality, comparison principle, stratified Lie group}
     \subjclass[2010]{35R03, 35S15}

     \begin{abstract}
     We propose analogues of Green's and Picone's identities for the $p$-sub-Laplacian on stratified Lie groups. In particular, these imply a generalised D\'iaz-Sa\'a inequality for the $p$-sub-Laplacian on stratified Lie groups. Using these 
     we derive a comparison principle and uniqueness of a positive solution to nonlinear hypoelliptic equations on general stratified Lie groups extending to the setting of the general stratified Lie groups previously known results on Euclidean and Heisenberg groups. 
     \end{abstract}
     \maketitle

\section{Introduction}

A stratified Lie group can be defined in many different equivalent ways (see e.g. \cite{FR} for the Lie algebra point of view). 
We follow the definition in \cite{BLU07}, that is, 
	a Lie group $\mathbb{G}=(\mathbb{R}^{N},\circ)$ is called a stratified Lie group (or a homogeneous Carnot group) if it satisfies the following two conditions:
	
	(i) For natural numbers $N_{1}+...+N_{r}=N$
	the decomposition $\mathbb{R}^{N}=\mathbb{R}^{N_{1}}\times...\times\mathbb{R}^{N_{r}}$ holds, and
	for each $\lambda>0$ the dilation $\delta_{\lambda}: \mathbb{R}^{N}\rightarrow \mathbb{R}^{N}$
	given by
	$$\delta_{\lambda}(x)\equiv\delta_{\lambda}(x^{(1)},...,x^{(r)}):=(\lambda x^{(1)},...,\lambda^{r}x^{(r)})$$
	is an automorphism of the group $\mathbb{G}.$ Here $x^{(k)}\in \mathbb{R}^{N_{k}}$ for $k=1,...,r.$
	
	(ii) Let $N_{1}$ be as in (i) and let $X_{1},...,X_{N_{1}}$ be the left invariant vector fields on $\mathbb{G}$ such that
	$X_{k}(0)=\frac{\partial}{\partial x_{k}}|_{0}$ for $k=1,...,N_{1}.$ 
	Then the H{\"o}rmander condition
	$${\rm rank}({\rm Lie}\{X_{1},...,X_{N_{1}}\})=N$$
	holds for every $x\in\mathbb{R}^{N},$ i.e. the iterated commutators
	of $X_{1},...,X_{N_{1}}$ span the Lie algebra of $\mathbb{G}.$

That is, we say that the triple $\mathbb{G}=(\mathbb{R}^{N},\circ, \delta_{\lambda})$ is a stratified Lie group (or a stratified group, in short). 
The number $r$ above is called the step of $\mathbb{G}$ and the left invariant vector fields $X_{1},...,X_{N_{1}}$ are called
the (Jacobian) generators of $\mathbb{G}$. 
The number
$$Q=\sum_{k=1}^{r}kN_{k}$$
is called the homogeneous dimension of $\mathbb{G}$.
We will also use the notation
$$\nabla_{\mathbb{G}}:=(X_{1},\ldots, X_{N_{1}})$$
for the (horizontal) gradient. We also recall that the standard Lebesgue measure $dx$ on $\mathbb{R}^{N}$ is the Haar measure for $\mathbb{G}$.
Let $\Omega\subset \mathbb{G}$ be an open set. The notation $u\in C^{1}(\Omega)$ means 
 $\nabla_{\mathbb{G}} u\in C(\Omega)$. 
 
We will also use the functional spaces $S^{1,p}(\Omega)=\{u:\;\Omega\rightarrow \mathbb{R};\;u,\,|\nabla_{\mathbb{G}} u|\in L^{p}(\Omega)\}$. 
 Moreover, let us consider the following functional 
$$J_{p}(u):=\left(\int_{\Omega}|\nabla_{\mathbb{G}} u|^{p}dx\right)^{\frac{1}{p}},$$
then we define the functional class $\overset{\circ}{S}^{1,p}(\Omega)$ to be the completion of $C^{1}_{0}(\Omega)$ in the norm generated by $J_{p}$ (see, e.g. \cite{CDG}). 
The operator 
\begin{equation}\label{pLap}
\mathcal{L}_{p}f:=\nabla_{\mathbb{G}}\cdot(|\nabla_{\mathbb{G}}f|^{p-2}\nabla_{\mathbb{G}}f),\quad 1<p<\infty,
\end{equation}
is called the subelliptic $p$-Laplacian or, in short, $p$-sub-Laplacian. 
Throughout this paper $\Omega\subset\mathbb{G}$ will be an admissible domain, that is, an open set $\Omega\subset\mathbb{G}$ is called an {\em admissible domain} if it is bounded and if its boundary $\partial\Omega$ is piecewise smooth and simple i.e., it has no self-intersections.
The condition for the boundary to be simple amounts to $\partial\Omega$ being orientable.
In $\Omega\subset\mathbb{G}$ we consider the following nonlinear Dirichlet boundary value problem for the $p$-sub-Laplacian
\begin{equation}\label{bvp}
\left\{
\begin{array}{ll}
	-\mathcal{L}_{p}u=F(x,u), \quad\text{in}\; \Omega,\\
	u=0, \quad\text{on}\; \partial\Omega. \\
\end{array}
\right.
\end{equation}
In this note we assume that:
\begin{itemize}
\item[(a)] The function $F:\Omega\times\mathbb{R}\rightarrow\mathbb{R}$ is a positive, bounded and measurable function and there exists a positive constant $C>0$ such that $F(x,\rho)\leq C(\rho^{p-1}+1)$ for a.e. $x\in\Omega$.
\item[(b)] The function $\rho\mapsto \frac{F(x,\rho)}{\rho^{p-1}}$
is strictly decreasing on $(0,\infty)$ for  a.e. $x\in\Omega$.    
\end{itemize}

As usual, a (weak) solution of \eqref{bvp}
means a function $u\in \overset{\circ}{S}^{1,p}(\Omega)\cap L^{\infty}(\Omega)$ such that
$$
\int_{\Omega}|
\nabla_{\mathbb{G}}u|^{p-2}(\mathcal{\widetilde{\nabla} }u)\phi d\nu=\int_{\Omega}F(x,u)\phi d\nu$$
holds for all $\phi\in C^{\infty}_{0}(\Omega),$
where
$$\mathcal{\widetilde{\nabla} }u=\sum_{k=1}^{N_{1}}\left(X_{k}u\right)X_{k}.$$
Nowadays, in the abelian case, there is vast literature devoted to the  study of the boundary value problem \eqref{bvp}.  
In the analysis of subelliptic $p$-Laplacian on, e.g., Heisenberg type groups the boundary value problems of this type have been also intensively investigated.
In the non-abelian case some of very first results obtained regarding the boundary value problem \eqref{bvp} with $p=2$ are by Garofalo and Lanconelli \cite{GL}, where the authors obtained existence and nonexistence results using Rellich-Pohozaev type inequalities.
Since then a number of studies have been devoted to this subject and most of them are on the Heisenberg group. See \cite{Bal}, \cite{Bia}, \cite{Cap}, \cite{Dou}, \cite{DT16}, \cite{GN}, \cite{LU}, \cite{Ruzhansky-Suragan:Kohn-Laplacian} and \cite{Ugu} as well as references therein.  
To the best of our knowledge, these results have not been extended to the general stratified Lie groups. Therefore, the aim of this short note is to extend to the setting of the stratified Lie groups previously known results on Euclidean and Heisenberg groups. To reach the desired results first one tries to obtain related Picone type identities$\backslash$inequalities (see, e.g. \cite{AP}, \cite{Bal} and \cite{Dou}), and we follow these ideas in the proofs. However, stratified group adapted $p$-sub-Laplacian Green identities based on our previous paper \cite{Ruzhansky-Suragan:Layers} play key roles in some calculations. Thus, we discuss $p$-sub-Laplacian Green identities and their applications in Section \ref{SEC:1}. In Section \ref{SEC:2} we derive versions of Picone's equality and inequality,
and give their applications, namely, proofs of a comparison principle as well as uniqueness of a positive solution of the Dirichlet boundary value problem for the $p$-sub-Laplacian \eqref{bvp}. 

\section{$p$-sub-Laplacian Green's identity and consequences}
\label{SEC:1}

Let $Q\geq 3$ be the homogeneous dimension of a stratified Lie group $\mathbb{G}$ and let $d\nu$ be the volume element on $\mathbb{G}$. Note that the Lebesque measure on $\mathbb R^{N}$ is the Haar measure for $\mathbb{G}$ (see, e.g. \cite[Proposition 1.3.21]{BLU07} or \cite[Proposition 1.6.6]{FR}). The notations $X=\{X_{1},...,X_{N_{1}}\}$ are left-invariant vector fields in the first stratum of $\mathbb{G}$, and $\langle X_{k}, d\nu\rangle$ is the natural pairing between vector fields and differential forms, more precisely, we have
\begin{equation}\label{Xjdnu0}
\langle X_{k}, d\nu(x)\rangle=\bigwedge_{j=1,j\neq k}^{N_{1}}dx_{j}^{(1)}\bigwedge_{l=2}^{r}
\bigwedge_{m=1}^{N_{l}}\theta_{l,m},
\end{equation}
where
\begin{equation}\label{theta0}
\theta_{l,m}=-\sum_{k=1}^{N_{1}}a_{k,m}^{(l)}(x^{(1)},\ldots,x^{(l-1)})
dx_{k}^{(1)}+dx_{m}^{(l)},\,\,l=2,\ldots,r,\,\,m=1,\ldots,N_{l},
\end{equation}
and $a_{k,m}^{(l)}$ is a $\delta_{\lambda}$-homogeneous polynomial of degree $l-1$
such that
\begin{equation}\label{Xk0}
X_{k}=\frac{\partial}{\partial x_{k}^{(1)}}+
\sum_{l=2}^{r}\sum_{m=1}^{N_{l}}a_{k,m}^{(l)}(x^{(1)},...,x^{(l-1)})
\frac{\partial}{\partial x_{m}^{(l)}},
\end{equation}
see \cite{Ruzhansky-Suragan:Layers}.
As mentioned in the introduction throughout this paper we assume that 	a domain $\Omega\subset\mathbb{G}$ is an admissible domain. We recall the following divergence formula for $X_{k}$'s.
\begin{prop}[\cite{Ruzhansky-Suragan:Layers}]
	\label{stokes}
 Let $f_{k}\in C^{1}(\Omega)\bigcap C(\overline{\Omega}),\,k=1,\ldots,N_{1}$. Then for each $k=1,\ldots,N_{1},$ we have
	\begin{equation}\label{EQ:S1}
	\int_{\Omega}X_{k}f_{k}d\nu=
	\int_{\partial\Omega}f_{k} \langle X_{k},d\nu\rangle.
	\end{equation}
	Consequently, we also have
	\begin{equation}\label{EQ:S2}
	\int_{\Omega}\sum_{k=1}^{N_{1}}X_{k}f_{k}d\nu=
	\int_{\partial\Omega}\sum_{k=1}^{N_{1}}f_{k} \langle X_{k},d\nu\rangle.
	\end{equation}
\end{prop}
As a consequences of the above Divergence formula we obtain the following analogue of Green's first identity for the $p$-sub-Laplacian. This version was proved for the sub-Laplacian ($p=2$) in \cite{Ruzhansky-Suragan:Layers}, and now we extend it to all $1<p<\infty.$
\begin{prop}[Green's first identity]\label{green1} Let $1<p<\infty.$
 Let $v\in C^{1}(\Omega)\bigcap C(\overline{\Omega})$ and  $u\in C^{2}(\Omega)\bigcap C^{1}(\overline{\Omega})$. Then
	\begin{equation} \label{g1}
	\int_{\Omega}\left((|\nabla_{\mathbb{G}}u|^{p-2}\mathcal{\widetilde{\nabla}}v) u+v\mathcal{L}_{p}u\right) d\nu=\int_{\partial\Omega}|\nabla_{\mathbb{G}}u|^{p-2}v\langle \mathcal{\widetilde{\nabla} }u,d\nu\rangle,
	\end{equation}
	where $\mathcal{L}_{p}$ is the p-sub-Laplacian on $\mathbb{G}$ and
	$$\mathcal{\widetilde{\nabla} }u=\sum_{k=1}^{N_{1}}\left(X_{k}u\right)X_{k}.$$
\end{prop}

\begin{proof}[Proof of Proposition \ref{green1}]
	Let $f_{k}=v|\nabla_{\mathbb{G}}u|^{p-2} X_{k}u,$ then
	
	$$\sum_{k=1}^{N_{1}}X_{k}f_{k}=
	(|\nabla_{\mathbb{G}}u|^{p-2}\mathcal{\widetilde{\nabla} }v) u+v\mathcal{L}_{p}u.$$
	By integrating both sides of this equality over $\Omega$ and using  Proposition \ref{stokes} we obtain
	\begin{multline*}
	\int_{\Omega}\left((|\nabla_{\mathbb{G}}u|^{p-2}\mathcal{\widetilde{\nabla} }v) u+v\mathcal{L}_{p}u\right) d\nu$$$$=\int_{\Omega}
	\sum_{k=1}^{N_{1}}X_{k}f_{k}d\nu \\
	=\int_{\partial\Omega}\sum_{k=1}^{N_{1}}\langle f_{k}X_{k},d\nu\rangle
	=\int_{\partial\Omega}\sum_{k=1}^{N_{1}}\langle v|\nabla_{\mathbb{G}}u|^{p-2} X_{k}uX_{k},d\nu\rangle
	=\int_{\partial\Omega}|\nabla_{\mathbb{G}}u|^{p-2}v\langle \mathcal{\widetilde{\nabla} }u,d\nu\rangle,
	\end{multline*}
	completing the proof.
\end{proof}
When $v=1$ Proposition \ref{green1} implies the following analogue of Gauss' mean value formula for
$p$-harmonic functions:
\begin{cor}\label{COR:v0}
	If $\mathcal{L}_{p}u=0$ in an admissible domain $\Omega\subset\mathbb{G}$, then
	$$\int_{\partial\Omega}|\nabla_{\mathbb{G}}u|^{p-2}\langle \mathcal{\widetilde{\nabla} }u,d\nu\rangle=0.$$
\end{cor}

As consequences of Proposition \ref{green1} we obtain the following uniqueness results for
not only the $p$-sub-Laplacian Dirichlet boundary value problem, but also other boundary value problems of different types, such as Neumann, Robin, or mixed types of conditions on different parts of the boundary.

We should mention that most of
the following results are known and can be proved by other methods too, but using given Green's first identity in Proposition \ref{green1} their proofs become elementary.

\begin{cor}\label{uniqueness}
	The Dirichlet boundary value problem
	\begin{equation}\label{Lu0}
	\mathcal{L}_{p}u(x)=0,\,\, x\in\Omega\subset\mathbb{G},
	\end{equation}
	\begin{equation}\label{u0}
	u(x)=0,\,\, x\in\partial\Omega,
	\end{equation}
	has the unique trivial solution
	$u\equiv0$ in the class of functions $C^{2}(\Omega)\bigcap C^{1}(\overline{\Omega})$.
\end{cor}

\begin{proof}[Proof of Corollary \ref{uniqueness}.]
	Set $v=\overline{u}$ in \eqref{g1}, then by \eqref{Lu0} and \eqref{u0} we get
	$$\int_{\Omega}(|\nabla_{\mathbb{G}}u|^{p-2}\mathcal{\widetilde{\nabla} }\overline{u}) u d\nu=\int_{\Omega}\left((|\nabla_{\mathbb{G}}u|^{p-2}
	\mathcal{\widetilde{\nabla} }\overline{u}) u+\overline{u}\mathcal{L}_{p}u\right)d\nu$$$$
	=\int_{\partial\Omega}|\nabla_{\mathbb{G}}u|^{p-2}\overline{u}\langle \mathcal{\widetilde{\nabla} }u,d\nu\rangle=0.$$
	Therefore
	$$\int_{\Omega}\sum_{k=1}^{N_{1}}|X_{k}u|^{p}d\nu=0,$$
	that is,
	$X_{k}u=0,\,\, k=1,...,N_{1}.$ Since any element of a Jacobian basis of $\mathbb{G}$ is represented by Lie brackets of $\{X_{1},...,X_{N_{1}}\}$, we obtain that
	$u$ is a constant, so $u\equiv0$ on $\Omega$ by \eqref{u0}.
\end{proof}

This has the following simple extension to (nonlinear) Schr\"odinger operators:

\begin{cor}\label{schr}
	Let $q:\mathbb{C}\times\Omega\rightarrow \mathbb{R}$ be a non-negative bounded function.
	Then the Dirichlet boundary value problem for the (nonlinear) Schr\"odinger equation
	\begin{equation}\label{Su0}
	-\mathcal{L}_{p}u(x)+q(u,x)u(x)=0,\,\, x\in\Omega\subset\mathbb{G},
	\end{equation}
	\begin{equation}\label{s0}
	u(x)=0,\,\, x\in\partial\Omega,
	\end{equation}
	has the unique trivial solution
	$u\equiv0$ in the class of functions $C^{2}(\Omega)\bigcap C^{1}(\overline{\Omega})$.
\end{cor}

\begin{proof}[Proof of Corollary \ref{schr}.]
	As in proof of Corollary  \ref{uniqueness} using Green's identity, from \eqref{Su0} and \eqref{s0} we obtain
	\begin{multline*}
	\int_{\Omega}(|\nabla_{\mathbb{G}}u|^{p-2}
	\mathcal{\widetilde{\nabla} }\overline{u}) u d\nu=\int_{\Omega}\left((|\nabla_{\mathbb{G}}u|^{p-2}
	\mathcal{\widetilde{\nabla} }\overline{u})  u
	+\overline{u}\mathcal{L}_{p}u\right)d\nu
	-\int_{\Omega}q(u,y)|u(y)|^{2}d\nu \\
	=\int_{\partial\Omega}|\nabla_{\mathbb{G}}u|^{p-2}\overline{u}\langle \mathcal{\widetilde{\nabla} }u,d\nu\rangle
	-\int_{\Omega}q(u,y)|u(y)|^{2}d\nu=
	-\int_{\Omega}q(u,y)|u(y)|^{2}d\nu.
	\end{multline*}
	Therefore,
	$$0\leq \int_{\Omega}\sum_{k=1}^{N_{1}}|X_{k}u|^{p}d\nu=-\int_{\Omega}q(u,y)|u(y)|^{2}d\nu\leq 0,$$
	that is, $u\equiv0$.
\end{proof}

Similarly, we obtain the following fact for the new measure-type von Neumann boundary conditions:

\begin{cor}\label{neumann}
	The boundary value problem
	\begin{equation}\label{Nu0}
	\mathcal{L}_{p}u(x)=0,\,\, x\in\Omega\subset\mathbb{G},
	\end{equation}
	\begin{equation}\label{n0}
	\sum_{j=1}^{N_{1}}X_{j}u\langle X_{j} ,d\nu\rangle=0\;\textrm{ on }\;\partial\Omega,
	\end{equation}
	has a solution
	$u\equiv \text{const}$ in the class of functions $C^{2}(\Omega)\bigcap C^{1}(\overline{\Omega})$.
\end{cor}

\begin{proof}[Proof of Corollary \ref{neumann}.]
	Set $v=\overline{u}$ in \eqref{g1}, then by \eqref{Nu0} and \eqref{n0} we get
	
	$$\int_{\Omega}(|\nabla_{\mathbb{G}}u|^{p-2}\mathcal{\widetilde{\nabla} }\overline{u}) u d\nu=\int_{\Omega}\left((|\nabla_{\mathbb{G}}u|^{p-2}
	\mathcal{\widetilde{\nabla} }\overline{u}) u+\overline{u}\mathcal{L}_{p}u\right)d\nu
	=\int_{\partial\Omega}|\nabla_{\mathbb{G}}u|^{p-2}\overline{u}\langle \mathcal{\widetilde{\nabla} }u,d\nu\rangle=0.$$
	Therefore,
	$$\int_{\Omega}\sum_{k=1}^{N_{1}}|X_{k}u|^{p}d\nu=0,$$
	that is,
	$X_{k}u=0,\,\, k=1,...,N_{1}.$ Since any element of a Jacobian basis of $\mathbb{G}$ is represented by Lie brackets of $\{X_{1},...,X_{N_{1}}\}$, we obtain that
	$u$ is a constant.
\end{proof}

In the same way one can consider the Robin-type boundary conditions as follows.
\begin{cor}\label{robin}
	Let $a_{j}:\partial\Omega\rightarrow {\mathbb R},\,\,j=1,...,N_{1},$ be bounded functions
	such that the measure
	\begin{equation}\label{m1}
	\sum_{j=1}^{N_1}a_j\langle X_{j} ,d\nu\rangle\geq 0
	\end{equation}
	is non-negative on $\partial\Omega$.
	Then the boundary value problem
	\begin{equation}\label{Ru0}
	\mathcal{L}_{p}u(x)=0,\,\, x\in\Omega\subset\mathbb{G},
	\end{equation}
	\begin{equation}\label{r0}
	\sum_{j=1}^{N_{1}}(a_{j}u+X_{j}u)\langle X_{j} ,d\nu\rangle=0\; \textrm{ on }\; \partial\Omega,\,\,
	\end{equation}
	has a solution
	$u\equiv const$ in the class of functions $C^{2}(\Omega)\bigcap C^{1}(\overline{\Omega})$.
\end{cor}

\begin{proof}[Proof of Corollary \ref{robin}.]
	Set $v=\overline{u}$ in \eqref{g1}, then by \eqref{Ru0} and \eqref{r0} we get
	\begin{multline}\label{Robin1}
\int_{\Omega}(|\nabla_{\mathbb{G}}u|^{p-2}\mathcal{\widetilde{\nabla} }\overline{u}) u d\nu=\int_{\Omega}\left((|\nabla_{\mathbb{G}}u|^{p-2}
\mathcal{\widetilde{\nabla} }\overline{u}) u+\overline{u}\mathcal{L}_{p}u\right)d\nu
 \\
=\int_{\partial\Omega}|\nabla_{\mathbb{G}}u|^{p-2}\overline{u}\langle \mathcal{\widetilde{\nabla} }u,d\nu\rangle
	=-\int_{\partial\Omega}|\nabla_{\mathbb{G}}u|^{p-2}|u|^{2}\sum_{j=1}^{N_{1}}a_{j}\langle X_{j} ,d\nu\rangle,
	\end{multline}
	that is,
	$$0\leq\int_{\Omega}\sum_{k=1}^{N_{1}}|X_{k}u|^{p}d\nu=-
	\int_{\partial\Omega}|\nabla_{\mathbb{G}}u|^{p-2}|u|^{2}\sum_{j=1}^{N_{1}}a_{j}\langle X_{j} ,d\nu\rangle\leq 0.$$
	Therefore
	$$\int_{\Omega}\sum_{k=1}^{N_{1}}|X_{k}u|^{p}d\nu=0$$
	and
	$$\int_{\partial\Omega}|\nabla_{\mathbb{G}}u|^{p-2}|u|^{2}\sum_{j=1}^{N_{1}}a_{j}\langle X_{j} ,d\nu\rangle=0.$$
	As above the first equality implies that $u$ is a constant. This proves the claim.
\end{proof}

We can also consider boundary value problems where Dirichlet or Robin boundary conditions are imposed on different parts of the boundary. The proof is similar to the above cases.

\begin{cor}\label{DN}
	Let $a_{j}:\partial\Omega\rightarrow {\mathbb R},\,\,j=1,...,N_{1},$ be bounded functions
	such that the measure
	\begin{equation}\label{m1a}
	\sum_{j=1}^{N_1}a_j\langle X_{j} ,d\nu\rangle\geq 0
	\end{equation}
	is non-negative on $\partial\Omega$.
	Let $\partial\Omega_{1}\subset \partial\Omega$, $\partial\Omega_{1}\not= \emptyset$ and  $\partial\Omega_{2}:= \partial\Omega\backslash\partial\Omega_{1}$.
	Then the boundary value problem
	\begin{equation}\label{DNu0}
	\mathcal{L}_{p}u(x)=0,\,\, x\in\Omega\subset\mathbb{G},
	\end{equation}
	\begin{equation}\label{dn0}
	u=0 \;\textrm{ on }\; \partial\Omega_{1},
	\end{equation}
	\begin{equation}\label{dn01}
	\sum_{j=1}^{N_{1}}(a_{j}u+X_{j}u)\langle X_{j} ,d\nu\rangle=0 \;\textrm{ on }\; \partial\Omega_{2},\,\,
	\end{equation}
	has the unique trivial solution
	$u\equiv 0$ in the class of functions $C^{2}(\Omega)\bigcap C^{1}(\overline{\Omega})$.
\end{cor}

\medskip
As a consequence of the Green's first identity \eqref{g1} we obtain the following analogue of Green's second identity for the $p$-sub-Laplacian:
\begin{prop}[Green's second identity]
	\label{green2}
	Let $1<p<\infty.$ Let $\Omega\subset\mathbb{G}$ be an admissible domain. Let $u,v\in C^{2}(\Omega)\bigcap C^{1}(\overline{\Omega}).$ Then
	\begin{multline}\label{g2}
	\int_{\Omega}
	\left(u\mathcal{L}_{p}v-v\mathcal{L}_{p}u
	+(|\nabla_{\mathbb{G}}v|^{p-2}-|\nabla_{\mathbb{G}}u|^{p-2})(
	\mathcal{\widetilde{\nabla}}v) u\right)d\nu\\
	=\int_{\partial\Omega}(|\nabla_{\mathbb{G}}v|^{p-2}u\langle\widetilde{\nabla}  v,d\nu\rangle-|\nabla_{\mathbb{G}}u|^{p-2}v\langle \widetilde{\nabla}  u,d\nu\rangle).
	\end{multline}
\end{prop}

\begin{proof}[Proof of Proposition \ref{green2}.]
	Rewriting \eqref{g1} we have
$$\int_{\Omega}\left((|\nabla_{\mathbb{G}}v|^{p-2}
\mathcal{\widetilde{\nabla}}u) v+u\mathcal{L}_{p}v\right) d\nu=\int_{\partial\Omega}|\nabla_{\mathbb{G}}v|^{p-2}u\langle \mathcal{\widetilde{\nabla} }v,d\nu\rangle,$$

		$$\int_{\Omega}\left((|\nabla_{\mathbb{G}}u|^{p-2}\mathcal{\widetilde{\nabla}}v) u+v\mathcal{L}_{p}u\right) d\nu=\int_{\partial\Omega}|\nabla_{\mathbb{G}}u|^{p-2}v\langle \mathcal{\widetilde{\nabla} }u,d\nu\rangle.$$
	By subtracting the second identity from the first one
	and using $$(\mathcal{\widetilde{\nabla} }u) v=(\mathcal{\widetilde{\nabla} }v) u$$
	we obtain the desired result.
\end{proof}

It is known that the sub-Laplacian ($p=2$) has a unique fundamental solution
$\varepsilon$ on $\mathbb{G}$ (see \cite{Fol75}),
$$
\mathcal{L}_{2}\varepsilon=\delta,
$$
and $\varepsilon(x,y)=\varepsilon(y^{-1}x)$ is homogeneous of degree $-Q+2$ and represented in the form
\begin{equation}\label{fundsol}
\varepsilon(x,y)= [d(x,y)]^{2-Q},
\end{equation}
where $d$ is the $\mathcal{L}$-gauge.

One of the largest classes of the stratified Lie groups, for which the fundamental solution of the $p$-sub-Laplacian is expressed explicitly are polarizable Carnot groups. A Lie group $\mathbb{G}$ is called a polarizable Carnot group if the $\mathcal{L}$-gauge $d$ satisfies the following $\infty$-sub-Laplacian equality  
$$\mathcal{L_{\infty}}d:=
\frac{1}{2}\nabla_{\mathbb{G}}|\nabla_{\mathbb{G}}d|^{2}\cdot \nabla_{\mathbb{G}}d=0\quad\text{in}\;\mathbb{G}\backslash \{0\}.$$
In the paper \cite{BT02} it was proved that if $\mathbb{G}$ is a polarizable Carnot group, then the fundamental solutions of the $p$-sub-Laplacian \eqref{pLap} are given by the explicit formulae 
\begin{equation}
\varepsilon_{p}:=\left\{
\begin{array}{ll}
c_{p}d^{\frac{p-Q}{p-1}}, \quad\text{if}\; p\not=Q,\\
-c_{Q}\log d, \quad\text{if}\; p=Q. \\
\end{array}
\right.
\end{equation} 
As usual, the Green identities are still valid for functions with (weak) singularities provided we can approximate them by smooth functions. 
Thus, for example,  for $x\in\Omega$ in a polarizable Carnot group, taking $v=1$ and $u(y)=\varepsilon_{p}(x,y)$ we record
the following consequence of Proposition \ref{green1}:
	If $\Omega$ is an admissible domain of a polarizable Carnot group $\mathbb{G}$, and $x\in\Omega$, then
	$$\int_{\partial\Omega}|\nabla_{\mathbb{G}}\varepsilon_{p}
	|^{p-2}\langle \mathcal{\widetilde{\nabla} }\varepsilon_{p}(x,y),d\nu(y)\rangle=1,$$
	where $\varepsilon_{p}$ is the fundamental solution of the $p$-sub-Laplacian.

For the polarizable Carnot groups putting the fundamental solution $\varepsilon_{p}$ instead of $v$ in \eqref{g2} we get the following representation type formulae:

\begin{itemize}
	\item
	Let $u\in C^{2}(\Omega)\bigcap C^{1}(\overline{\Omega})$. Then for $x\in\Omega$ we have
	\begin{multline}\label{rep}
	u(x)=\int_{\Omega}\varepsilon_{p}\mathcal{L}_{p}u
	-(|\nabla_{\mathbb{G}}\varepsilon_{p}|
	^{p-2}-|\nabla_{\mathbb{G}}u|^{p-2})(
	\mathcal{\widetilde{\nabla}}\varepsilon_{p}) ud\nu\\
	+\int_{\partial\Omega}(|\nabla_{\mathbb{G}}\varepsilon_{p}
	|^{p-2}u\langle\widetilde{\nabla}  \varepsilon_{p},d\nu\rangle-|\nabla_{\mathbb{G}}u|^{p-2}\varepsilon_{p}\langle \widetilde{\nabla}  u,d\nu\rangle).
	\end{multline}
\item
	Let $u\in C^{2}(\Omega)\bigcap C^{1}(\overline{\Omega})$ and $\mathcal{L}_{p}u=0$ on $\Omega$, then for $x\in\Omega$ we have
	\begin{multline}\label{rep}
	u(x)=\int_{\Omega}
	(|\nabla_{\mathbb{G}}u|^{p-2}-|\nabla_{\mathbb{G}}\varepsilon_{p}|
	^{p-2})(
	\mathcal{\widetilde{\nabla}}\varepsilon_{p}) ud\nu\\
	+\int_{\partial\Omega}(|\nabla_{\mathbb{G}}\varepsilon_{p}
	|^{p-2}u\langle\widetilde{\nabla}  \varepsilon_{p},d\nu\rangle-|\nabla_{\mathbb{G}}u|^{p-2}\varepsilon_{p}\langle \widetilde{\nabla}  u,d\nu\rangle).
	\end{multline}
\item
	Let $u\in C^{2}(\Omega)\bigcap C^{1}(\overline{\Omega})$ and
	\begin{equation}\label{D1}
	u(x)=0,\,\, x\in\partial\Omega,
	\end{equation}
	then
	\begin{equation}
	u(x)=\int_{\Omega}\varepsilon_{p}\mathcal{L}_{p}u
	-(|\nabla_{\mathbb{G}}\varepsilon_{p}|
	^{p-2}-|\nabla_{\mathbb{G}}u|^{p-2})(
	\mathcal{\widetilde{\nabla}}\varepsilon_{p}) ud\nu-\int_{\partial\Omega}|\nabla_{\mathbb{G}}u|^{p-2}\varepsilon_{p}\langle \widetilde{\nabla}  u,d\nu\rangle.
	\end{equation}
\item
	Let $u\in C^{2}(\Omega)\bigcap C^{1}(\overline{\Omega})$ and
	\begin{equation}\label{n1}
	\sum_{j=1}^{N_{1}}X_{j}u\langle X_{j} ,d\nu\rangle=0 \;\textrm{ on }\; \partial\Omega,
	\end{equation}
	then
	\begin{equation}\label{rep}
	u(x)=\int_{\Omega}\varepsilon_{p}\mathcal{L}_{p}u
	-(|\nabla_{\mathbb{G}}\varepsilon_{p}|
	^{p-2}-|\nabla_{\mathbb{G}}u|^{p-2})(
	\mathcal{\widetilde{\nabla}}\varepsilon_{p}) ud\nu\\
	+\int_{\partial\Omega}|\nabla_{\mathbb{G}}\varepsilon_{p}
	|^{p-2}u\langle\widetilde{\nabla}  \varepsilon_{p},d\nu\rangle
	.
	\end{equation}
\end{itemize}

Of course, these representation formula expressions hold in general stratified Lie groups provided $\varepsilon_{p}$ exists. However, according to the meaning of the classical cases, one should know the fundamental solution in an explicit form, so we have focused on the polarizable groups. Note that there are stratified Lie groups (other than polarizable ones) in which the fundamental solution of, say, the sub-Laplacian can be expressed explicitly (see \cite[Section 6]{BT02}).

\section{$p$-sub-Laplacian Picone's identity and  consequences}
\label{SEC:2}
By keeping in mind the stratified group discussions from the introduction for any set $\Omega\subset \mathbb{G}$ and a locally Lipschitz function $f:\mathbb{R}^{+}\rightarrow \mathbb{R}^{+}$ such that $ (p-1)|f(t)|^{\frac{p-2}{p-1}}\leq f'(t)$ a.e. in $\mathbb{R}^{+}$ with $1<p<\infty$  we introduce the notations 

\begin{equation}
L(u,v):=|\nabla_{\mathbb G}u|^{p}-p\frac{|u|^{p-2}u}{f(v)}\nabla_{\mathbb G}u\cdot\nabla_{\mathbb G}v |\nabla_{\mathbb G}v|^{p-2}+\frac{f'(v)|u|^{p}}{f^{2}(v)}|\nabla_{\mathbb G}v|^{p} 
\end{equation} 
and
\begin{equation}
R(u,v):=|\nabla_{\mathbb G}u|^{p}-\nabla_{\mathbb G}\left (\frac{|u|^{p}}{f(v)}\right)|\nabla_{\mathbb G}v|^{p-2}\nabla_{\mathbb G}v  
\end{equation} 
a.e. in $\Omega$.
Then we have the following stratified Lie group version of Picone's identity. 
\begin{lem}\label{thm1}
For any set $\Omega\subset \mathbb{G}$ and all $1<p<\infty$ we have 
$L(u,v)=R(u,v)\geq 0$ a.e. in $\Omega$, where $u$ and $v$ are differentiable real-valued  functions.
\end{lem}
\begin{proof} [Proof of Lemma \ref{thm1}.]
	A direct calculation shows that
\begin{align*}
\nabla_{\mathbb G}\left (\frac{|u|^{p}}{f(v)}\right) & =  \frac{pf(v)|u|^{p-2}u \nabla_{\mathbb G}u-f'(v)|u|^{p}\nabla_{\mathbb G}v}{f^{2}(v)}
 \\  & = \frac{p|u|^{p-2}u \nabla_{\mathbb G}u}{f(v)}-
\frac{f'(v)|u|^{p}\nabla_{\mathbb G}v}{f^{2}(v)}.
\end{align*}	
Thus,
\begin{align*}R(u,v) & =|\nabla_{\mathbb G}u|^{p}-\nabla_{\mathbb G}\left (\frac{|u|^{p}}{f(v)}\right)|\nabla_{\mathbb G}v|^{p-2}\nabla_{\mathbb G}v
\\ & =|\nabla_{\mathbb G}u|^{p}- \frac{p|u|^{p-2}u}{f(v)}|\nabla_{\mathbb G}v|^{p-2} \nabla_{\mathbb G}u\cdot\nabla_{\mathbb G}v+
\frac{f'(v)|u|^{p}}{f^{2}(v)}|\nabla_{\mathbb G}v|^{p}\\ & = L(u,v).
\end{align*}
Now it remains to show the nonnegativity of $R(u,v)$. We have

\begin{equation*} \frac{p|u|^{p-2}u}{f(v)}|\nabla_{\mathbb G}v|^{p-2} \nabla_{\mathbb G}u\cdot\nabla_{\mathbb G}v\leq \frac{p|u|^{p-1}}{f(v)}|\nabla_{\mathbb G}v|^{p-1} |\nabla_{\mathbb G}u|,\end{equation*}
further by using the Young inequality we arrive at
\begin{equation*} \frac{p|u|^{p-2}u}{f(v)}|\nabla_{\mathbb G}v|^{p-2} \nabla_{\mathbb G}u\cdot\nabla_{\mathbb G}v\leq |\nabla_{\mathbb G}u|^{p}+(p-1)\frac{|u|^{p}|\nabla_{\mathbb G}v|^{p} }{f^{\frac{p}{p-1}}(v)}.\end{equation*}
It follows that 
\begin{equation*} \frac{f'(v)|u|^{p}|\nabla_{\mathbb G}v|^{p}}{f^{2}(v)} -(p-1)\frac{|u|^{p}|\nabla_{\mathbb G}v|^{p} }{f^{\frac{p}{p-1}}(v)}\leq R(u,v) .\end{equation*}
Since by the definition $(p-1)|f(t)|^{\frac{p-2}{p-1}}\leq f'(t)$, 
this means $0\leq R(u,v).$ 
\end{proof}

As a consequence of the Harnack inequality for the general hypoelliptic equation (see \cite[Theorem 3.1]{CDG}) we have the following strong maximum principle for the $p$-sub-Laplacian. The proofs of both Lemma \ref{strongmaxprinc} and  \ref{Hardyineq} are similar to the case of Heisenberg groups (see, \cite{Dou} for more details).  
\begin{lem}\label{strongmaxprinc}
	Let  $\Omega\subset \mathbb{G}$ be a bounded open set, $1<p\leq Q,$ and  let $F:\Omega\times\mathbb{R}\rightarrow\mathbb{R}$ be a measurable function such that $|F(x,\rho)|\leq C(\rho^{p-1}+1)$ for all $\rho>0.$
	Let $u\in \overset{\circ}{S}^{1,p}(\Omega)$ be a nonnegative solution of   
\begin{equation}
\left\{
\begin{array}{ll}
-\mathcal{L}_{p}u=F(x,u), \quad\text{in}\; \Omega\\
u=0, \quad\text{on}\; \partial\Omega. \\
\end{array}
\right.
\label{EQ:fs}
\end{equation}
Then $u\equiv0$ or $u>0$ in $\Omega$.
\end{lem}
\begin{proof} [Proof of Lemma \ref{strongmaxprinc}.]
Since $u\in \overset{\circ}{S}^{1,p}(\Omega)$ by using the Harnack inequality \cite[Theorem 3.1]{CDG} for $1<p\leq Q$ there exists a constant $C_{R}$ such that
\begin{equation*}
\underset{B(0,R)}{\sup}\{u(x)\}\leq C_{R}\underset{B(0,R)}{\inf}\{u(x)\}
\end{equation*}
for any quasi-ball $B(0,R).$ 
This means $u\equiv0$ or $u>0$ in $B(0,R)$, that is, $u\equiv0$ or $u>0$ in $\Omega$.
\end{proof}

\begin{lem} \label{Hardyineq}
	Let $\Omega\subset \mathbb{G}$ be a bounded open set and let $v\in \overset{\circ}{S}^{1,p}(\Omega)$ be such that $v\geq\epsilon>0$. Then for all $p>1$ and $u\in C^{\infty}_{0}(\Omega)$ we have  
	\begin{equation}\label{ineq_Hardytype}
\int_{\Omega}\frac{|u|^{p}}{f(v)} (-\mathcal{L}_{p}v)dx\leq
\int_{\Omega}|\nabla_{\mathbb G}u|^{p}dx.
	\end{equation}
\end{lem}

As above here $f:\mathbb{R}^{+}\rightarrow \mathbb{R}^{+}$ is a locally Lipschitz function such that $ (p-1)|f(t)|^{\frac{p-2}{p-1}}\leq f'(t)$ a.e. in $\mathbb{R}^{+}$ with $1<p<\infty$.

\begin{proof}[Proof of Lemma \ref{Hardyineq}.]
By the density argument we can choose $v_{k}\in C^{1}_{0}(\Omega),\;k=1,2,\ldots,$ 
such that $v_{k}>\frac{\epsilon}{2}$ in $\Omega$ and $v_{k}\rightarrow v$
a.e. in $\Omega$.
By using Lemma \ref{thm1} we obtain that
$$0\leq \int_{\Omega}R(u,v_{k})dx,$$
for each $k$.
That is, 
		\begin{equation*}
	\int_{\Omega}\frac{|u|^{p}}{f(v_{k})} (-\mathcal{L}_{p}v_{k})dx\leq
	\int_{\Omega}|\nabla_{\mathbb G}u|^{p}dx.
	\end{equation*}
In addition, from the fact that $\mathcal{L}_{p}$ is a continuous operator from $\overset{\circ}{S}^{1,p}(\Omega)$ to $S^{-1,p^{\prime}}(\Omega),$
$p^{\prime}=\frac{p}{p-1},$ (cf. \cite[Theorem A.0.6]{Per} 
) we have 
$\mathcal{L}_{p}v_{k}\rightarrow \mathcal{L}_{p}v$ in $S^{-1,p^{\prime}}(\Omega)$
and $f(v_{k})\rightarrow f(v)$ pointwise since $f$ is a locally Lipschitz continuous 
function on $(0,\infty)$.
Thus, by the Lebesque dominated convergence theorem and using the fact that 
$f$ is an increasing function on $(0,\infty)$, for any  $u\in C^{\infty}_{0}(\Omega)$ we arrive at 
 \begin{equation}
 \int_{\Omega}\frac{|u|^{p}}{f(v)} (-\mathcal{L}_{p}v)dx\leq
 \int_{\Omega}|\nabla_{\mathbb G}u|^{p}dx,
 \end{equation}
proving \eqref{ineq_Hardytype}.
\end{proof}

By using Lemma \ref{Hardyineq} above we prove the following generalised Picone inequality: 
\begin{thm}\label{P_inequality}
	Let  $\Omega\subset \mathbb{G}$ be a bounded open set and let  $g:\Omega\times\mathbb{R}\rightarrow\mathbb{R}$ be positive, bounded and measurable function such that $g(x,\rho)\leq C(\rho^{p-1}+1)$ for all $\rho>0.$ If functions $v,\,u\in \overset{\circ}{S}^{1,p}(\Omega)$ with $v(\not\equiv0)\geq0$ a.e. $\Omega\in \mathbb G$ are such that $-\mathcal{L}_{p}v=g(x,v)$, then 
	\begin{equation}\label{Picone_ineq}
	\int_{\Omega}\frac{|u|^{p}}{f(v)} (-\mathcal{L}_{p}v)dx\leq
	\int_{\Omega}|\nabla_{\mathbb G}u|^{p}dx,\quad 1<p<\infty.
	\end{equation}
\end{thm}
\begin{proof} [Proof of Theorem \ref{P_inequality}.]
	By Lemma \ref{strongmaxprinc} we have $v>0$ in $\Omega$. Let 
	$v_{k}(x):=v(x)+\frac{1}{k},\;k=1,2,\ldots$, then we have $\mathcal{L}_{p}v_{k}=\mathcal{L}_{p}v$ in $S^{-1,p^{\prime}}(\Omega)$,
	$v_{k}\rightarrow v$ a.e. in $\Omega$ and also 
	$f(v_{k})\rightarrow f(v)$ pointwise in $\Omega$.
	Let $u_{k}\in C^{\infty}_{0}(\Omega)$ be such that 	$u_{k}\rightarrow u$  in $\overset{\circ}{S}^{1,p}(\Omega)$. 
	For the functions $u_{k}$ and $v_{k}$ Lemma \ref{Hardyineq} gives
	\begin{equation*}
	\int_{\Omega}\frac{|u_{k}|^{p}}{f(v_{k})} (-\mathcal{L}_{p}v_{k})dx\leq
	\int_{\Omega}|\nabla_{\mathbb G}u_{k}|^{p}dx.
	\end{equation*}
	Now since $f(v_{k})\rightarrow f(v)$ pointwise by the Fatou lemma
	we arrive at
	\begin{equation*}
	\int_{\Omega}\frac{|u|^{p}}{f(v)} (-\mathcal{L}_{p}v)dx\leq
	\int_{\Omega}|\nabla_{\mathbb G}u|^{p}dx.
	\end{equation*}
	This completes the proof.
\end{proof}

As a consequence of the Picone inequality we have the following comparison type principle: 
\begin{thm} \label{com_principle}
	Let $\Omega$ be an admissible domain. Let $u,v\in \overset{\circ}{S}^{1,p}(\Omega)$ be real-valued functions such that 	
	\begin{equation}
	\left\{
	\begin{array}{ll}
	-\mathcal{L}_{p}u\geq F(x)u^{q}, \quad u>0 \quad\text{in}\; \Omega,\\
	-\mathcal{L}_{p}v\leq F(x)v^{q}, \quad v>0 \quad\text{in}\; \Omega.\\
	\end{array}
	\right.
	\label{comprin}
	\end{equation}
where $0<q<p-1,$ $F$ is a nonnegative function with $F\not\equiv0$. Then
$v\leq u$ a.e. in $\Omega$.
\end{thm}
\begin{proof} [Proof of Theorem \ref{com_principle}.] It follows from \eqref{comprin} that 
		\begin{equation*}
F(x)\left(\frac{u^{q}}{u^{p-1}}-\frac{v^{q}}{v^{p-1}}
\right)\leq \frac{-\mathcal{L}_{p}u}{u^{p-1}}+\frac{\mathcal{L}_{p}v}{v^{p-1}}.	
		\end{equation*}
	Multiplying both sides by $w=(v^{p}-u^{p})_{+}$ and integrating over 
	$\Omega$ we have 
		\begin{align}
	\int_{[v>u]}F(x)	\left(\frac{u^{q}}{u^{p-1}}-\frac{v^{q}}{v^{p-1}}
\right)wdx & =	\int_{\Omega}F(x)	\left(\frac{u^{q}}{u^{p-1}}-\frac{v^{q}}{v^{p-1}}
\right)wdx  \\ & \leq \int_{\Omega}\left(\frac{-\mathcal{L}_{p}u}{u^{p-1}}+\frac{\mathcal{L}_{p}v}{v^{p-1}}\right)wdx.	
	\label{43}\end{align}
In addition, a direct calculation gives 
	\begin{align*}
 \int_{\Omega}\left(\frac{-\mathcal{L}_{p}u}{u^{p-1}}+\frac{\mathcal{L}_{p}v}{v^{p-1}}\right)wdx & =
  \int_{\Omega}|\nabla_{\mathbb G}u|^{p-2}
  \nabla_{\mathbb G}u\cdot\nabla_{\mathbb G}
  \left(\frac{w}{u^{p-1}}\right)dx
	\\ &   - \int_{\Omega}|\nabla_{\mathbb G}v|^{p-2}
  \nabla_{\mathbb G}v\cdot\nabla_{\mathbb G}
  \left(\frac{w}{v^{p-1}}\right)dx	
 	\\ & =
 	 \int_{\Omega\cap[v>u]}|\nabla_{\mathbb G}u|^{p-2}
 	\nabla_{\mathbb G}u\cdot\nabla_{\mathbb G}
 	\left(\frac{v^{p}-u^{p}}{u^{p-1}}\right)dx
 	\\ & 	- \int_{\Omega\cap[v>u]}|\nabla_{\mathbb G}v|^{p-2}
 	\nabla_{\mathbb G}v\cdot\nabla_{\mathbb G}
 	\left(\frac{v^{p}-u^{p}}{v^{p-1}}\right)dx	
 	\\ & 
 		= \int_{\Omega\cap[v>u]}\left(|\nabla_{\mathbb G}u|^{p-2}
 	\nabla_{\mathbb G}u\cdot\nabla_{\mathbb G}
 	\left(\frac{v^{p}}{u^{p-1}}\right)-|\nabla_{\mathbb G}v|^{p}\right)dx
 	\\ & 	+ \int_{\Omega\cap[v>u]}\left(|\nabla_{\mathbb G}v|^{p-2}
 	\nabla_{\mathbb G}v\cdot\nabla_{\mathbb G}
 	\left(\frac{u^{p}}{v^{p-1}}\right)-|\nabla_{\mathbb G}u|^{p}\right)dx
 		\\ & = I_{1}+I_{2},
\end{align*}
where 
\begin{equation*}
I_{1}:=\int_{\Omega\cap[v>u]}\left(|\nabla_{\mathbb G}u|^{p-2}
\nabla_{\mathbb G}u\cdot\nabla_{\mathbb G}
\left(\frac{v^{p}}{u^{p-1}}\right)-|\nabla_{\mathbb G}v|^{p}\right)dx
\end{equation*}
	and 
	\begin{equation*}
	I_{2}:=\int_{\Omega\cap[v>u]}\left(|\nabla_{\mathbb G}v|^{p-2}
	\nabla_{\mathbb G}v\cdot\nabla_{\mathbb G}
	\left(\frac{u^{p}}{v^{p-1}}\right)-|\nabla_{\mathbb G}u|^{p}\right)dx.
	\end{equation*}

We have 	
	\begin{align*} I_{1} & =\int_{\Omega\cap[v>u]}|\nabla_{\mathbb G}u|^{p-2}
	\nabla_{\mathbb G}u\cdot\nabla_{\mathbb G} \left(
	\frac{v^{p}}{u^{p-1}}\right)dx-\int_{\Omega\cap[v>u]}|\nabla_{\mathbb G}v|^{p}dx
	\\ & =-\int_{\Omega\cap[v>u]}
	\frac{v^{p}}{u^{p-1}}\mathcal{L}_{p}udx-\int_{\Omega\cap[v>u]}|\nabla_{\mathbb G}v|^{p}dx
	 \leq 0 .
\end{align*}
In the last line we have used Green's first identity \eqref{g1} and the Picone inequality \eqref{Picone_ineq}.
Similarly, we see that  
$ I_{2} \leq 0.$
Thus, we obtain 
 \begin{equation*}\int_{\Omega}\left(\frac{-\mathcal{L}_{p}u}{u^{p-1}}+\frac{\mathcal{L}_{p}v}{v^{p-1}}\right)wdx\leq 0.
\end{equation*}
Consequently, \eqref{43} implies that 
 \begin{equation*}\int_{\Omega\cap [v>u]}F(x)
 \left(\frac{u^{q}}{u^{p-1}}+\frac{v^{q}}{v^{p-1}}\right)(v^{p}-u^{p})dx\leq 0.
\end{equation*}
On the other hand, we have 
$$0\leq F(x)
\left(\frac{u^{q}}{u^{p-1}}+\frac{v^{q}}{v^{p-1}}\right)$$
for $[v>u]$. This means 
$|[v>u]|=0.$
\end{proof}

As another consequence of the generalised Picone inequality we obtain the following D\'iaz-Sa\'a inequality
on stratified Lie groups.
\begin{thm}\label{DS} 	Let $\Omega$ be an admissible domain. Let functions $g_{1}$ and $g_{2}$ satisfy the assumption of Theorem \ref{P_inequality}. If functions $u_{1},\,u_{2}\in \overset{\circ}{S}^{1,p}(\Omega)$ with $u_{1},\,u_{2}(\not\equiv0)\geq0$ a.e. $\Omega\in \mathbb G$ are such that $-\mathcal{L}_{p}u_{1}=g_{1}(x,u_{1})$
	and $-\mathcal{L}_{p}u_{2}=g_{2}(x,u_{2})$, then  
	\begin{equation*}
0\leq\int_{\Omega}
\left(\frac{-\mathcal{L}_{p}u_{1}}{u_{1}^{p-1}}+\frac{\mathcal{L}_{p}u_{2}}{u_{2}^{p-1}}\right) (u_{1}^{p}-
u_{2}^{p}) dx 
.
\end{equation*}
\end{thm}
\begin{proof} [Proof of Theorem \ref{DS}.]
	Let functions $u_{1}$ and  $u_{2}$ satisfy the assumptions. Then by the inequality \eqref{Picone_ineq} with $f(u)=u^{p-1}$ as well as for $u_{1}$ and $u_{2}$ we have 
		\begin{equation*}
	\int_{\Omega}\frac{|u_{1}|^{p}}{u^{p-1}_{2}} (-\mathcal{L}_{p}u_{2})dx\leq
	\int_{\Omega}|\nabla_{\mathbb G}u_{1}|^{p}dx.
	\end{equation*}
Using Green's first identity \eqref{g1} we get   	
		\begin{equation}\label{pr16}
	0\leq\int_{\Omega}
	\left(\frac{-\mathcal{L}_{p}u_{1}}{u_{1}^{p-1}}+\frac{\mathcal{L}_{p}u_{2}}{u_{2}^{p-1}}\right) u_{1}^{p} dx.
	\end{equation}
Again, by the inequality \eqref{Picone_ineq} we have 
		\begin{equation*}
	\int_{\Omega}\frac{|u_{2}|^{p}}{u^{p-1}_{1}} (-\mathcal{L}_{p}u_{1})dx\leq
	\int_{\Omega}|\nabla_{\mathbb G}u_{2}|^{p}dx.
	\end{equation*}
As above, this implies 	
		\begin{equation}\label{pr17}
	0\leq\int_{\Omega}
	\left(\frac{\mathcal{L}_{p}u_{1}}{u_{1}^{p-1}}-\frac{\mathcal{L}_{p}u_{2}}{u_{2}^{p-1}}\right)
	u_{2}^{p} dx.
	\end{equation}
Now the combination of \eqref{pr16} and \eqref{pr17} completes the proof. 
\end{proof}
Finally, we prove the following theorem on uniqueness of a positive solution of
\begin{equation}\label{bvp_l}
\left\{
\begin{array}{ll}
-\mathcal{L}_{p}u=F(x,u), \quad\text{in}\; \Omega,\\
u=0, \quad\text{on}\; \partial\Omega, \\
\end{array}
\right.
\end{equation}
where $\Omega$ is an admissible domain.
Here we recall the assumptions on $F(x,u)$:
\begin{itemize}
	\item[(a)] The function $F:\Omega\times\mathbb{R}\rightarrow\mathbb{R}$ is a positive, bounded and measurable function and there exists a positive constant $C>0$ such that $F(x,\rho)\leq C(\rho^{p-1}+1)$ for a.e. $x\in\Omega$.
	\item[(b)] The function $\rho\mapsto \frac{F(x,\rho)}{\rho^{p-1}}$
	is strictly decreasing on $(0,\infty)$ for  a.e. $x\in\Omega$.    
\end{itemize}

\begin{thm} \label{main} There exists at most one positive weak solution to \eqref{bvp_l} for $1<p\leq Q.$	
\end{thm}
\begin{proof} [Proof of Theorem \ref{main}.] Suppose that $u_{1}$ and $u_{2}$ be two different ($u_{1}\not\equiv u_{2}$) non-negative solutions of \eqref{bvp_l}. By using the strong maximum principle in Lemma \ref{strongmaxprinc} for the $p$-sub-Laplacian we have $u_{1}>0$ and $u_{2}>0$ in $\Omega$. By Theorem \ref{DS} we have 
		\begin{equation*}
	0\leq\int_{\Omega}
	\left(\frac{-\mathcal{L}_{p}u_{1}}{u_{1}^{p-1}}+\frac{\mathcal{L}_{p}u_{2}}{u_{2}^{p-1}}\right) (u_{1}^{p}-
	u_{2}^{p}) dx.
	\end{equation*}
On the other hand, according to the assumption (b) we have the strict inequality 
\begin{equation*}
\int_{\Omega}
\left(\frac{F(x,u_{1})}{u_{1}^{p-1}}-\frac{F(x,u_{2})}{u_{2}^{p-1}}\right) (u_{1}^{p}-
u_{2}^{p}) dx<0.
\end{equation*}
Since 
\begin{equation*}
\int_{\Omega}
\left(\frac{-\mathcal{L}_{p}u_{1}}{u_{1}^{p-1}}+\frac{\mathcal{L}_{p}u_{2}}{u_{2}^{p-1}}\right) (u_{1}^{p}-
u_{2}^{p}) dx
=\int_{\Omega}
\left(\frac{F(x,u_{1})}{u_{1}^{p-1}}-\frac{F(x,u_{2})}{u_{2}^{p-1}}\right) (u_{1}^{p}-
u_{2}^{p}) dx,
\end{equation*}
this contradicts that both $u_{1}$ and $u_{2}$ ($u_{1}\not\equiv u_{2}$) are non-negative solutions of \eqref{bvp_l}.
\end{proof}

\end{document}